\documentclass[11pt,a4paper]{amsart}
\usepackage[latin1]{inputenc}
\usepackage[english]{babel}
\usepackage{amsmath,amssymb,amsthm}
\usepackage[dvipsnames]{xcolor}

\usepackage{enumerate}
\usepackage{cite}
\usepackage{mathtools}

\theoremstyle{plain}
\newtheorem{theorem}{Theorem}[section]
\newtheorem{lemma}[theorem]{Lemma}

\newtheorem{prop}[theorem]{Proposition}

\newtheorem{proposition}[theorem]{Proposition}

\theoremstyle{definition}

\newtheorem{example}[theorem]{Example}

\newcommand{\abs}[1]{\lvert#1\rvert}

\newtheorem{remark}[theorem]{Remark}

\let\phi\varphi
\let\epsilon\varepsilon
\newcommand{\norm}[1]{\Vert#1\Vert}
\newcommand{\bignorm}[1]{\bigl\Vert#1\bigr\Vert}
\newcommand{\Bignorm}[1]{\Bigl\Vert#1\Bigr\Vert}

 \newcommand{\A}{\mbox{${\mathcal A}$}}

   \newcommand{\D}{\mbox{${\mathcal D}$}}

   \renewcommand{\H}{\mbox{${\mathcal H}$}}

    \newcommand{\M}{\mbox{${\mathcal M}$}}
   \newcommand{\N}{\mbox{${\mathcal N}$}}

\begin{document}
\title{surjective separating maps on noncommutative $L^p$-spaces}

\author[C. Le Merdy]{Christian Le Merdy}

\address{Laboratoire de Math\'ematiques de Besan\c{c}on,
Universite Bourgogne Franche-Comt\'e, France}
\email{\texttt{christian.lemerdy@univ-fcomte.fr}}

\author[S. Zadeh]{Safoura Zadeh}
\address{
Max-Planck-Institut f\"{u}r Mathematik,
Vivatsgasse 7, 53111, Bonn, Germany \&\newline
\indent 
Mathematisches Institut, Westf\"{a}lische Wilhelms-Universit\"{a}t M\"{u}nster, Germany.
%Laboratoire de Math\'ematiques de Besan\c{c}on,
%Universite Bourgogne Franche-Comt\'e, France 
%\& Faculty of Graduate Studies, Dalhousie University, Canada.
}
\email{\texttt{jsafoora@gmail.com}}


\subjclass[2000]{}
%\footnotetext[1]{This work is supported by }
\begin{abstract}
Let $1\leq p<\infty$ and let $T\colon L^p(\M)\to L^p(\N)$ be a bounded map between noncommutative $L^p$-spaces. If $T$ is bijective and separating (i.e., for any $x,y\in L^p(\M)$ such that $x^*y=xy^*=0$, we have $T(x)^*T(y)=T(x)T(y)^*=0$), we prove the existence of decompositions $\M=\M_1\mathop{\oplus}\limits^\infty\M_2$, $\N=\N_1 \mathop{\oplus}\limits^\infty\N_2$ and maps $T_1\colon L^p(\M_1)\to L^p(\N_1)$, $T_2\colon L^p(\M_2)\to L^p(\N_2)$, such that $T=T_1+T_2$, $T_1$ has a direct Yeadon type factorisation and $T_2$ has an anti-direct Yeadon type factorisation. We further show that $T^{-1}$ is separating in this case. Next we prove that for any $1\leq p<\infty$ (resp. any $1\leq p\not=2<\infty$), a surjective separating map $T\colon L^p(\M)\to L^p(\N)$ is $S^1$-bounded (resp. completely bounded) if and only if there exists a decomposition $\M=\M_1 \mathop{\oplus}\limits^\infty\M_2$ such that $T|_{L^p({\tiny \M_1})}$ has a direct Yeadon type factorisation and $\M_2$ is subhomogeneous.
\end{abstract}
\maketitle
%%%%%%%%%%%%%%%%%%%%%%%%%%%%%%%%%%%%%%
\section{Introduction}\label{sec1}
%%%%%%%%%%%%%%%%%%%%%%%%%%%%%%%%%%%%
This paper deals with separating maps between noncommutative $L^p$-spaces, $1\leq p<\infty$. These operators were investigated recently in \cite{LMZ,LMZ2,HRW}, to which we refer for background, motivation and historical facts. Recall that a bounded map $T\colon L^p(\M)\to L^p(\N)$ between two noncommutative $L^p$-spaces is called separating if for any $x,y\in L^p(\M)$, the condition $x^*y=xy^*=0$ implies that $T(x)^*T(y)=T(x)T(y)^*=0$. It was shown in \cite[Proposition 3.11]{LMZ} and \cite[Theorem 3.3 \& Remark 3.4]{HRW} that $T\colon L^p(\M)\to L^p(\N)$ is separating if and only if there exists a $w^*$-continuous Jordan homomorphism $J\colon\M\to \N$, a positive operator $B$ affiliated with $\N$ and commuting with the range of $J$, as well as a partial isometry $w\in\N$ such that $w^*w=s(B)=J(1)$ and
$$
T(x)=wBJ(x),\qquad (x\in \M\cap L^p(\M)).
$$    
Such a factorization (which is necessarily unique) is called a Yeadon type factorization in \cite{LMZ,LMZ2}. We further say that $T$ admits a direct Yeadon type factorization if the Jordan homomorphism $J$ in this factorization is a $*$-homomorphism. It is proved in \cite[Proposition 4.4]{LMZ2} and \cite[Theorem 3.6]{HRW} that any separating map $T\colon L^p(\M)\to L^p(\N)$ with a direct Yeadon type factorization is necessarily completely bounded. It is also proved in \cite[Proposition 4.5]{LMZ2} that any such map is $S^1$-bounded (see Section 2 below for the definition). The main purpose of the present paper is to establish a form of converse of these results for surjective maps. More precisely, we prove the following characterizations.

\noindent {\bf Theorem.}
Let $1\leq p<\infty$, let $\M,\N$ be semifinite von Neumann algebras and let $T\colon L^p(\M)\to L^p(\N)$ be a surjective separating map. The following are equivalent :
\begin{itemize}
\item [(i)] $T$ is $S^1$-bounded;
\item [(ii)] There exists a direct sum decomposition $\M=\M_1\mathop{\oplus}\limits^\infty \M_2$ such that the restriction of $T$ to $L^p(\M_1)$ has a direct Yeadon type factorization and $\M_2$ is subhomogeneous.
\end{itemize}
Moreover if $p\not=2$, then (ii) is also equivalent to :
\begin{itemize}
\item [(iii)] $T$ is completely bounded.
\end{itemize}
These results will be proved in Section 4. We also provide an example showing that the surjectivity assumption cannot be dropped. In section 3, we establish a general decomposition result for bijective separating maps which plays a key role in the above characterization results. We prove in passing that the inverse of any bijective separating map is separating as well. Section 2 is preparatory.
 
%%%%%%%%%%%%%%%%%%%%%%%%%%%%%%%%%%%%%%
\section{Background}
%%%%%%%%%%%%%%%%%%%%%%%%%%%%%%%%%%%%%%%
In this section we recall some necessary background on semifinite noncommutative $L^p$-spaces and subhomogeneous von Neumann algebras.

Let $\M$ be a semifinite von Neumann algebra with a normal semifinite faithful trace $\tau_{\tiny\M}$. Assume that $\M\subset B(\H)$ acts 
on some Hilbert space $\H$. Let $L^0(\M)$ 
denote the $*$-algebra of all closed densely defined (possibly unbounded)
operators on $\H$, which are
$\tau_{_{\tiny{\M}}}$-measurable. 
Then for any $1\leq p<\infty$, the noncommutative $L^p$-space associated with $\mathcal{M}$ can be defined as
$$
L^p(\mathcal{M}):=\bigl\{x\in L^0(\mathcal{M})
\,:\,\tau_{\tiny\M}(\lvert x\rvert^p)<\infty\bigr\}.
$$
We set 
$\|x\|_p:=\tau_{\tiny\M}(\left\lvert x\rvert^p\right)^{\frac{1}{p}}$
for any 
$x\in L^p(\M)$. Then $L^p(\M)$ equipped with $\norm{\,\cdotp}_p$
is a Banach space.  The reader may consult \cite{JX,PX,T} and the references therein for details and further properties.

We let $S^p$, $1\leq p<\infty$, denote the noncommutative $L^p$-space built upon $B(\ell^2)$ with its usual trace; this is in fact the Schatten $p$-class of operators on $\ell^2$. For any $m\geq 1$, we let $S^p_m$ denote the Schatten $p$-class of $m\times m$
matrices. Whenever $E$ is an operator space, we let $S^p_m[E]$ denote the $E$-valued Schatten space introduced in \cite[Chapter1]{P5}.

Recall that we may identify $L^p(\M\otimes M_m)$ with $L^p(\M)\otimes S^p_m$ in a natural way. Let $\N$ be, possibly, another semifinite von Neumann algebra. We say that an operator $T:L^p(\M)\to L^p(\N)$ is completely bounded if there exists a constant $K\geq0$ such that 
$$
 \| T\otimes I_{S^p_m}\colon L^p(\M\otimes M_m)\to L^p(\N\otimes M_m)\|\leq K,
 $$
 for any $m\geq1$. In this case, the completely bounded norm of $T$ is the smallest such uniform bound and is denoted by $\|T\|_{cb}$. We further say that $T$ is a complete isometry if $T\otimes I_{S^p_m}$ is an isometry for any $m\geq1$.
 
In \cite[Section 3]{LMZ2}, we introduced $S^1$-valued noncommutative $L^p$-spaces, which naturally extend previous constructions from \cite{P5,J}. We recall this definition here. 

For $1\leq p<\infty$, the $S^1$-valued noncommutative $L^p$-space, $L^p(\M;S^1)$, is the space of all infinite matrices $[x_{ij}]_{i,j\geq1}$ in $L^p(\M)$ for which there exist families $(a_{ik})_{i,k\geq1}$ and $(b_{kj})_{k,j\geq1}$ in $L^{2p}(\M)$ such that $\sum_{i,k}a_{ik}a_{ik}^*$ and $\sum_{k,j}b_{kj}^*b_{kj}$ converge in $L^p(\M)$ and for all $i,j\geq1$,
$$
x_{ij}=\sum_{k=1}^\infty a_{ik}b_{kj}.
$$
We equip $L^p(\M;S^1)$ with the following norm
\begin{align}\label{defS1Norm}
\|[x_{ij}]\|_{L^p(\tiny\M;S^1)}=\inf\left\{\Bignorm{\sum_{i,k=1}^\infty a_{ik}a_{ik}^*}_p^{\frac{1}{2}} \Bignorm{\sum_{k,j=1}^\infty b_{kj}^*b_{kj}}_p^{\frac{1}{2}}\right\},
\end{align}
where the infimum is taken over all families $(a_{ik})_{i,k\geq1}$ and $(b_{kj})_{k,j\geq1}$ as above.
The space $L^p(\M;S^1)$ endowed with this norm is a Banach space. 

For any integer $m\geq 1$, we let $L^p(\M;S^1_m)$ be the subspace of $L^p(\M;S^1)$ of matrices $[x_{ij}]_{i,j\geq1}$ with support in $\{1,\ldots,m\}^2$.

Following \cite[Definition 3.8]{LMZ2}, we say that a bounded operator $T:L^p(\M)\to L^p(\N)$ is $S^1$-bounded if there exists a constant $K\geq0$ such that 
$$
\|T\otimes I_{S^1_m}\colon L^p(\M;S^1_m)\longrightarrow L^p(\N;S^1_m)\|\leq K,
$$
for any $m\geq1$. In this case, the $S^1$-bounded norm of $T$ is the smallest such uniform bounded and is denoted by $\|T\|_{S^1}$. We further say that $T:L^p(\M)\to L^p(\N)$ is an $S^1$-isometry if for each $m\geq1$,
 $T\otimes I_{S^1_m}$ is an isometry.
 
 We proved in \cite{LMZ2} that for any $n\geq1$, $L^p(M_n; S^1_m)=S^p_n[S^1_m]$ isometrically. Further, if $\M, \N$ are hyperfinite, then $T:L^p(\M)\to L^p(\N)$ is $S^1$-bounded if and only if it is regular in the sense of \cite{P2}.
 
We note that any direct sum $\M=\M_1\mathop{\oplus}\limits^{\infty}\M_2$ induces isometric identifications $L^p(\M)=L^p(\M_1)\mathop{\oplus}\limits^p L^p(\M_2)$ and $L^p(\M;S^1)=L^p(\M_1;S^1)\mathop{\oplus}\limits^p L^p(\M_2;S^1)$ (see \cite[Lemma 5.2]{LMZ2} for the last identification). 
 
Recall that a $C^*$-algebra $\A$ is called subhomogeneous of degree $\leq N$ if all irreducible representations of $\mathcal{A}$ are of maximum dimension $N$. If $\A$ is subhomogeneous of degree $\leq N$, for some $N$, we simply say that $\A$ is subhomogeneous. It is well-known (see for example \cite[Theorem 7.1.1]{Ru}) that $\M$ is a subhomogeneous von Neumann algebra of degree $\leq N$ if and only if there exist $r\geq1$, integers $1\leq n_1\leq n_2\leq \ldots\leq n_r\leq N$ and abelian von Neumann algebras $L^\infty(\Omega_1),\ldots, L^\infty(\Omega_r)$ such that
 \begin{align}\label{n-sub}
 \M\simeq \mathop\oplus\limits_{1\leq j\leq r}^\infty L^\infty(\Omega_j;M_{n_j}).
 \end{align}

If a von Neumann algebra $\M$ is not subhomogeneous of degree $\leq N$, it is well-known that there is a non zero $\ast$-homomorphism $\gamma:M_{N+1}\to\M$. Lemma \ref{l2} below makes this more explicit in the semifinite case.
\begin{lemma}\label{l2}
	Let $\M$ be a semifinite von Neumann algebra and let  $N\geq1$. If $\M$ is not subhomogeneous of degree $\leq N$, then there is a complete isometry from $S^p_{N+1}$ into $L^p(\M)$ that is also an $S^1$-isometry.
\end{lemma}
\begin{proof}
Let $\M=\M_1\mathop{\oplus}\limits^\infty \M_2$ be the
direct sum decomposition of $\M$ 
into a type I summand $\M_1$ and a type 
II summand $\M_2$ (see e.g. \cite[Section 5]{Ta}). 

Assume that $\M_2\neq\{0\}$. Following the same lines as in \cite[Lemma 2.3]{LMZ2}, there is a projection $e$ in $\M_2$, a trace preserving von Neumann algebra identification 
\begin{align}\label{matrixAlg}
\M_2\simeq M_{N+1}\overline{\otimes}(e\M_2 e)
\end{align}
and a finite trace projection $\epsilon$ in $e\M_2 e$ such that the mapping 
$$\gamma:M_{N+1}\to\M_2\subset\M;\qquad\gamma(a)=a\otimes\epsilon
$$ 
is a non zero $\ast$-homomorphism taking values in $L^1(\M)$, and therefore $L^p(\M)$.

For every $[a_{ij}]_{1\leq i,j\leq m}$ in $S^p_{N+1}\otimes S^p_m$ we have that
$$
\|[a_{ij}\otimes\epsilon]\|_{L^p(\tiny\M_2\otimes M_m)}=\|\epsilon\|_p\|[a_{ij}]\|_{L^p(M_{N+1}\otimes M_m)},
$$
and therefore $\|\epsilon\|_p^{-1}\gamma$ is a complete isometry from $S^p_{N+1}$ into $L^p(\M)$. By \cite[Lemma 5.1]{LMZ2},
$$
\|[a_{ij}\otimes\epsilon]\|_{L^p(\tiny\M_2;\; S^1_m)}=\|\epsilon\|_p\|[a_{ij}]\|_{S^p_{N+1}[S^1_m]},
$$
and therefore $\|\epsilon\|_p^{-1}\gamma$ is also an $S^1$-isometry from $S^p_{N+1}$ into $L^p(\M)$.

If $\M_{2}=\{0\}$, then $\M$ is of type I. Since $\M$ is not subhomogeneous of degree $\leq N$, it follows from \cite[Theorem V.1.27]{Ta} that there exist a Hilbert space $\H$ with $\dim(\H)\geq N+1$ and an abelian von Neumann algebra $W$ such that $\M$ contains $B(\H)\overline{\otimes}W$ as a summand. Using this summand instead of \eqref{matrixAlg} and arguing as above we obtain the result in this case as well.
\end{proof}
 
%%%%%%%%%%%%%%%%%%%%%%%%%%%%%%%%%%%%%%%%%%%%%%%%%%%%%%%
\section{bijective separating maps and their inverses}
%%%%%%%%%%%%%%%%%%%%%%%%%%%%%%%%%%%%%%%%%%%%%%%%%%%%%%

The goal of this section is to provide a decomposition for bijective separating maps that facilitates their study. We apply this decomposition to show that the inverse of a bijective separating map is separating as well.\\

First we recall some terminologies and results that we will use. A Jordan homomorphism between von Neumann algebras $\M$ and $\N$ is a linear map $J:\mathcal{M}\to\mathcal{N}$ such that
	$$
	J(x^*)=J(x)^*\quad\text{and}\quad J(xy+yx)=J(x)J(y)+J(y)J(x)
	$$
	for all $x$ and $y$ in $\mathcal{M}$. It is plain that $\ast$-homomorphisms and anti-$\ast$-homomorphisms are Jordan homomorphisms. In fact, every Jordan homomorphism is a sum of a $\ast$-homomorphism and an anti-$\ast$-homomorphism, as we recall here.

Let $J:\M\to\N$ be a Jordan homomorphism and let $\D\subset\N$ be the $w^*$-closed $C^*$-algebra generated by $J(\M)$. Then $J(1)$ is the unit of $\D$. By e.g.
\cite[Theorem 3.3]{S}, there exist projections 
$e$ and $f$ in the center of $\D$ such that $e+f=J(1)$, $x \mapsto J(x)e$ is a $*$-homomorphism, and $x\mapsto J(x)f$ is an anti-$*$-homomorphism. Let $\N_1 = e\N e$ and $\N_2 = f\N f$. Define $\pi:\M \to \N_1$ and $\sigma:\M \to \N_2$ by $\pi(x) = J(x)e$ and $\sigma(x) = J(x)f$, for all $x \in \M$. Then $J$ is valued in $\N_1 \mathop{\oplus}\limits^\infty\N_2$ and $J(x) = \pi(x) + \sigma(x)$, for all $x \in \M$.\\

Assume that $\M$ and $\N$ are semifinite von Neumann algebras and let $1\leq p<\infty$. In \cite{LMZ}, inspired by Yeadon's fundamental description of isometries between noncommutative $L^p$-spaces, we say that a bounded operator $T\colon L^p(\mathcal{M})\to L^p(\mathcal{N})$ 
has a Yeadon type factorization if there exist a $w^*$-continuous 
Jordan homomorphism $J\colon\mathcal{M}\to\mathcal{N}$, 
a partial isometry $w\in\mathcal{N}$, 
and a positive operator $B$ affiliated with $\mathcal{N}$, 
which satisfy the following conditions:
\begin{enumerate}[(a)]
\item $w^{\ast}w=J(1)=s(B)$, the support projection of $B$;
\item  every spectral projection of $B$ commutes with $J(x)$, for all $x\in\mathcal{M}$;
\item $T(x)=w BJ(x)$ for all $x\in\mathcal{M}\cap L^p(\mathcal{M})$.
\end{enumerate}
We call $(w, B, J)$ the Yeadon triple associated with $T$. This triple is unique. Following \cite{LMZ2}, if $J$ is a $\ast$-homomorphism (respectively, anti-$\ast$-homomorphism), we say that $T$ has a direct (respectively, anti-direct) Yeadon type factorization.  

Following \cite{LMZ}, we say that a bounded operator $T:L^p(\M)\to L^p(\N)$ is separating if for every $x,y\in L^p(\M)$ such that $x^*y=xy^*=0$, we have that $T(x)^*T(y)=T(x)T(y)^*=0$. The following characterization has a fundamental role in the study of separating maps.
\begin{theorem}{(\hspace{-0.01cm}\cite[Theorem 3.3]{HRW}, \cite[Theorem 3.5]{LMZ})}
	A bounded operator $T\colon L^p(\mathcal{M})\to L^p(\mathcal{N})$ admits a Yeadon type factorization 
if and only if it is separating.
\end{theorem}

It is easy to see that for a separating map $T:L^p(\M)\to L^p(\N)$ with Yeadon triple $(w,B,J)$, we have that 
	\begin{align}\label{repeated2}
		T(z^*)=wT(z)^* w\qquad(z\in L^p(\M)).
	\end{align}
	
Also, if $T$ has a direct (respectively, anti-direct) Yeadon type factorization, we get that
\begin{align}\label{repeated3}
		T(zm)=T(z)J(m)\ \left(\text{respectively, } T(mz)=T(z)J(m)\right),
	\end{align}   
for every $z\in L^p(\M)\text{ and } m\in\M$. 

\begin{remark}\label{Remark1}
	Let $T:L^p(\M)\to L^p(\N)$ be a separating map with Yeadon triple $(w,B,J)$. We observe that if $T$ is surjective, then $w$ is a unitary. Indeed on the one hand, we see that $T$ is valued in $wL^p(\N)$. Since $ww^*w=w$, this implies that $T$ is valued in $ww^*L^p(\N)$. Hence, if $T$ is surjective, we have $ww^*L^p(\N)=L^p(\N)$, which implies that $ww^*=1$. On the other hand, $T(x)=T(x)J(1)$, for any $x\in L^p(\M)$. Hence, $T$ is valued in $L^p(\N)J(1)$. Hence, if $T$ is surjective, we have $L^p(\N)J(1)=L^p(\N)$, which implies $w^*w=J(1)=1$.
\end{remark}

\begin{proposition}\label{3.1}
	Let $T:L^p(\M)\to L^p(\N)$ be a separating map that is bijective. Then there exist direct sum decompositions 
	$$
	\M=\M_1\mathop{\oplus}\limits^{\infty}\M_2,\qquad\text{and}\qquad \N=\N_1\mathop{\oplus}\limits^{\infty}\N_2,
	$$
	and bounded bijective separating maps $T_1:L^p(\M_1)\to L^p(\N_1)$ with a direct Yeadon type factorization and $T_2:L^p(\M_2)\to L^p(\N_2)$ with an anti-direct Yeadon type factorization such that $T=T_1+T_2$.
\end{proposition}

\begin{proof}
	Assume that $w=1$. Consider a decomposition for $J$, induced by central projections $e$ and $f$, as recalled above. As detailed in \cite[Remark 4.3]{LMZ2}, this induces a decomposition $\N=\N_1\mathop{\oplus}\limits^\infty\N_2$ and  separating maps
	$$
T_1\colon L^p(\M)\longrightarrow L^p(\N_1),\qquad T_1(x)= T(x)e,
$$
with Yeadon triple $(e, Be, \pi),$ and hence a direct Yeadon type factorization, and  
$$
T_2\colon L^p(\M)\longrightarrow L^p(\N_2),\qquad T_2(x)= T(x)f,
$$
with Yeadon triple  $(f, Bf, \sigma)$, and hence an anti-direct Yeadon type factorization, such that $T=T_1+T_2$.

Let $\M_1:=\ker(\sigma)$ and $\M_2:=\ker(\pi)$. Since $\M_1$ and $\M_2$ are $w^*$-closed ideals of $\M$, there exist central projections $\alpha,\beta\in\M$ such that $\M_1=\alpha\M$, and $\M_2=\beta\M$. Set $\M_3:=(1-\alpha)(1-\beta)\M$. Note that $\alpha\beta\in\ker(\sigma)\cap\ker(\pi)$, and therefore  $J(\alpha\beta)=0$. Since $T$ is one-to-one, by \cite[Remark 3.14(a)]{LMZ}, $J$ is one-to-one and therefore we must have that $\alpha\beta=0$. Hence, 
	$$
	1=\alpha+\beta+(1-\alpha)(1-\beta).
	$$
Consequently, $\M=\M_1\mathop{\oplus}\limits^{\infty}\M_2\mathop{\oplus}\limits^{\infty}\M_3$, and so we have the decomposition 
\begin{align*}
L^p(\M)=L^p(\M_1)\mathop{\oplus}\limits^p L^p(\M_2)\mathop{\oplus}\limits^p L^p(\M_3).
\end{align*}

The result will follow if we can show that 
\begin{align*}
L^p(\M_1)=\ker(T_2),\quad L^p(\M_2)=\ker(T_1)\quad\text{and}\quad \M_3=\{0\}.
\end{align*}

To see that $L^p(\M_1)\subseteq \ker(T_2)$, let $x\in\M_1\cap L^p(\M_1)$, then 
$$
T_2(x)=B\sigma(x)=0.
$$
Hence, $\M_1\cap L^p(\M_1)\subset \ker(T_2)$ and therefore $L^p(\M_1)\subset \ker(T_2)$. Now suppose that $x$ belongs to $\ker(T_2)$. For any $n\geq1$, let $p_n=\chi_{[-n,n]}(\abs{x^*})$, the projection associated with the indicator function of $[-n,n]$ in the Borel functional calculus of $\abs{x^*}$, and $x_n:=p_nx$. Then, using \eqref{repeated3}, we have that 
$$
T_2(x_n)=T_2(x) \sigma(p_n)=0.
$$
Hence, $B\sigma(x_n)=0$. Since $s(B)=1$, this implies that $\sigma(x_n)=0$, that is $x_n$ is in $\M_1$. Now because $x_n\to x$ in $L^p(\M)$, we obtain that $x$ belongs to $L^p(\M_1)$. Hence,
$$
L^p(\M_1)=\ker(T_2).
$$
Similarly, we can show that $L^p(\M_2)=\ker(T_2)$. 

Finally, we show that $\M_3=\{0\}$. Let $x\in L^p(\M)$. By surjectivity of $T$, there is $y$ in $L^p(\M)$ such that $T(y)=T_1(x)$. Writing $T(y)=T_1(y)+T_2(y)$, we obtain that $T_1(x-y)=0$ and $T_2(y)=0$, that is $x-y$ belongs to $\ker(T_1)=L^p(\M_2)$ and $y$ belongs to $\ker(T_2)=L^p(\M_1)$, thus $x$ belongs to  $L^p(\M_1)\mathop{\oplus}\limits^p L^p(\M_2)$. Hence, $\M_3=\{0\}$. This completes the proof in the case $w=1$.

In the general case, consider the map $\widetilde{T}:= w^*T(\cdot)$, which takes any $x\in\M\cap L^p(\M)$ to $BJ(x)$. By Remark \ref{Remark1}, $\widetilde{T}$ is also a bijective separating map. Its Yeadon triple is $(1,B,J)$. We may apply the above decomposition to the map $\widetilde{T}$ to obtain decompositions $\M=\M_1\mathop{\oplus}\limits^\infty\M_2$, $\N=\N_1\mathop{\oplus}\limits^\infty\N_2$ and bounded bijective separating maps $\widetilde{T_1}:L^p(\M_1)\to L^p(\N_1)$ with a direct Yeadon type factorization and $\widetilde{T_2}:L^p(\M_2)\to L^p(\N_2)$ with an anti-direct Yeadon type factorization such that $\widetilde{T}=\widetilde{T_1}+\widetilde{T_2}$. Since $w\widetilde{T}=T$, we obtain the result.
\end{proof}
%%%%%%%%%%%%%%%%%%%%%%%%%%%%

\begin{proposition}\label{lastprop}
Suppose that $T:L^p(\M) \to L^p(\N)$ is a bijective separating map, then
\begin{itemize}
	\item[(i)]  $T^{-1}:L^p(\N)\to L^p(\M)$ is separating.
	\item[(ii)] If $J:\M\to\N$ is the Jordan homomorphism associated with $T$, then $J$ is invertible and $J^{-1}:\N\to\M$ is the Jordan homomorphism associated with $T^{-1}$. 
\end{itemize} 
\end{proposition}
\begin{proof}
Using the decomposition given in Proposition \ref{3.1}, it is enough to show parts $(i)$ and $(ii)$ for a bijective separating map with a direct Yeadon type  factorization. So, throughout the proof we assume that this is the case. Note that by Remark \ref{Remark1}, $J(1)=1$. 

$(i)$ Suppose that $a,b\in L^p(\N)$ such that $a^*b=ab^*=0$. We show that $T^{-1}(a)^*T^{-1}(b)=T^{-1}(a)T^{-1}(b)^*=0$. Let $x=T^{-1}(a)$ and $y=T^{-1}(b)$. Set $p_n:=\chi_{[-n,n]}\left(\abs{y}\right)$, for any $n\geq1$. We have that 
\begin{align*}
T(x^*yp_n)B&=T(x^*)J(yp_n)B&\text{ by }(\ref{repeated3}) \\
&=	T(x^*)w^*T(yp_n)&\\
&=wT(x)^*T(yp_n)&\text{ by }(\ref{repeated2}) \\
&=wT(x)^*T(y)J(p_n)&\text{ by }(\ref{repeated3}) \\
&=wa^*bJ(p_n)=0.&
\end{align*}
	Since $s(B)=J(1)=1$, we obtain $T(x^*yp_n)=0$. Because $T$ is one-to-one, we have that $x^*yp_n=0$. Now, since $yp_n\to y$, we get that $x^*y=0$. A similar argument using $ab^*=0$ implies that $xy^*=0$. Hence $T^{-1}$ must be separating. \\

$(ii)$ By part $(i)$, $T^{-1}$ is separating. We let $J^{\prime}$ denote the Jordan homomorphism of its Yeadon triple. Let $e\in\N$ be a projection with finite trace. For any $y\in e\N e$, we have that $T^{-1}(y)=T^{-1}(e)J^{\prime}(y)$. Applying \eqref{repeated3}, we deduce that 
$$
y=TT^{-1}(y)=T\left(T^{-1}(e)J^{\prime}(y)\right)=TT^{-1}(e)\;JJ^{\prime}(y)= e\; JJ^{\prime}(y).
$$ Using the $w^*$-continuity of $J$ and $J^{\prime}$, and the $w^*$-density of the union of the $e\N e$, for $\tau_{\tiny\N}(e)<\infty$, we deduce that $y=JJ^{\prime}(y)$ for any $y\in \N$. By \cite[Remark 3.14(a)]{LMZ}, since $T$ is one-to-one, $J$ must be one-to-one. Hence, $J$ is invertible with $J^{-1}= J^{\prime}$. 
\end{proof}
\begin{remark}\label{remarkJ}
Part $(ii)$ of Proposition \ref{lastprop} shows that a separating invertible map $T:L^p(\M) \to L^p(\N)$ admits a direct Yeadon type factorization if and only if $T^{-1}$ does.	
\end{remark}

%%%%%%%%
\section{a characterization of completely/$S^1$-bounded surjective separating maps} 
%%%%%%%%%%%%%%%%%%%%%%%%%%%%%
In this section we show that a separating map can always be reduced to a one-to-one separating map and therefore we may confine ourself to the study of separating maps that are surjective rather than bijective. The goal of the section is to provide a characterization for surjective separating maps that are completely bounded (when $p\neq2$) or $S^1$-bounded. We show that the surjectivity assumption is essential.\\    

We require \cite[Propositions 4.4 \& 4.5]{LMZ2} later on in our arguments in this section. We recall the statements for convenience.  
\begin{proposition}\label{4.4}
Let $T\colon L^p(\M)\to L^p(\N)$ be a bounded operator with 
a direct Yeadon type factorization. Then $T$ is completely 
bounded and $\|T\|_{cb}=\|T\|$. 
\end{proposition}
\begin{proposition}\label{4.5}
Let $T\colon L^p(\M)\to L^p(\N)$ be a bounded operator with 
a direct Yeadon type factorization. Then $T$ is $S^1$-bounded and $\|T\|_{S^1}=\|T\|$.
\end{proposition}

\begin{lemma}\label{4.1}
	Let $T:L^p(\M)\to L^p(\N)$ be a separating map. Then there exists a direct sum decomposition 
	$
	\M=\M_0\mathop{\oplus}\limits^{\infty}\widetilde{\M}
	$
such that $\ker{(T)}=L^p(\M_0)$.
\end{lemma}
\begin{proof}
	Let $T:L^p(\M)\to L^p(\N)$ be a separating map and $J:\M\to\N$ be the Jordan homomorphism associated with $T$ via its Yeadon type factorization. Let $\M_0:=\ker(J)$. 
	Then $\M_0$ is an ideal. Since $J$ is $w^*$-continuous, $\M_0$ is $w^*$-closed. Hence we have a direct sum decomposition
$$\M=\M_0\mathop{\oplus}\limits^\infty\widetilde{\M}.$$
It is clear that $L^p(\M_0)\subset\ker{T}$. Further $J|_{\widetilde{\tiny\M}}$ is one-to-one. By \cite[Remark 3.14(a)]{LMZ} this implies that $T|_{L^p(\widetilde{\tiny\M})}$ is one-to-one. This yields the result.
\end{proof}
For any von Neumann algebra $\M$, we let $\M^{op}$ denote its opposite von Neumann algebra. Recall that the underlying dual Banach space structure and involution on $\M^{op}$ are the same as on $\M$ but the product of $x$ and $y$ is defined by $yx$ rather than $xy$. Note that the Banach spaces $L^p(\M)$ and $L^p(\M^{op})$ are the same. It is evident that, for von Neumann algebras $\M$ and $\N$, $J:\M\to\N$ is a $*$-homomorphism if and only if $$J^{op}:\M^{op}\to\N;\quad x\mapsto J(x),$$ is an anti-$*$-homomorphism. Hence, a separating map $T:L^p(\M)\to L^p(\N)$ has a direct Yeadon type factorization if and only if 
$$
T^{op}:L^p(\M^{op})\to L^p(\N);\quad x\mapsto T(x),
$$
has an anti-direct Yeadon type factorization.\\

Lemma \ref{l32} below is the principal ingredient of our characterization theorems. Its proof relies on the relation between the completely bounded norm or $S^1$-norm of the identity map
$$
I^{op}:L^p(\M)\to L^p(\M^{op})
$$
and the norms of the transformations
$$
[x_{ij}]_{1\leq i,j\leq m}\mapsto [x_{ji}]_{1\leq i,j\leq m}
$$
either on $L^p(\M\otimes M_m)$ or on $L^p(\M;S^1_m)$, in particular in the specific case when $\M=M_n$. We will use the fact that for any $n\geq1$, we have $L^p(M_n\otimes M_m)\simeq S^p_m[S^p_n]$, isometrically, provided that $S^p_n$ is equipped with the operator space structure given in \cite{P5}.

Let $t_m$ denote the transposition map on scalar $m\times m$ matrices. Assume that $\M$ is semifinite. The map $$I_{\tiny{\M}^{op}} \otimes t_m: \M^{op} \otimes M_m \to \M^{op} \otimes M_m^{op} $$ is a trace preserving
$*$-homomorphism, and so
$$
I_{L^p(\tiny{\M}^{op})} \otimes t_m
\colon L^p(\M^{op} \otimes M_m )\longrightarrow 
L^p(\M^{op}\otimes M_m^{op})
$$
is an isometry. Moreover  $\M^{op}\otimes M_m^{op} =(\M\otimes M_m)^{op}$, hence
$L^p(\M^{op}\otimes M_m^{op})
=L^p(\M\otimes M_m)$ isometrically. For any $[x_{ij}]_{1\leq i,j\leq m}$ in $L^p(\M)\otimes S^p_m$, since $ I_{L^p(\tiny{\M}^{op})} \otimes t_m$ maps
$[x_{ij}]$ to $[x_{ji}]$, we get that
\begin{equation}\label{OpTrCb}
\bignorm{[x_{ij}]}_{L^p(\tiny{\M}^{op}\otimes M_m)}
= \bignorm{[x_{ji}]}_{L^p(\tiny{\M}\otimes M_m)}.
\end{equation}
We now show that similarly, for any $[x_{ij}]_{1\leq i,j\leq m}$ in $L^p(\M)\otimes S^1_m$,
\begin{equation}\label{OpTrS1}
\bignorm{[x_{ij}]}_{L^p(\tiny{\M}^{op};\; S^1_m)}
= \bignorm{[x_{ji}]}_{L^p(\tiny{\M};\; S^1_m)}.
\end{equation}
To verify the identity (\ref{OpTrS1}), assume that $\norm{[x_{ij}]}_{L^p(\tiny{\M}^{op};\; S^1_m)}<1$.
Taking into account the opposite product and \eqref{defS1Norm}, we can write
$$
x_{ij}= \sum_k b_{kj}a_{ik}
$$
for some $a_{ik},b_{kj}$ in $L^{2p}(\M)$ such that 
$\sum_{i,k} a_{ik}^*a_{ik}$ and $\sum_{k,j} b_{kj}b_{kj}^*$
have norm $<1$ in $L^{p}(\M)$. This exacly means that 
$\norm{[x_{ji}]}_{L^p(\tiny{\M};\; S^1_m)}<1$. This shows the
inequality $\geq$ in (\ref{OpTrS1}). Reversing the argument
we find the other inequality.

Identities (\ref{OpTrCb}) and (\ref{OpTrS1}), respectively, imply 
\begin{equation}\label{OpCbNorm}
\bignorm{I^{op}\colon L^p(\M)\longrightarrow 
L^p(\M^{op})}_{cb} =\,\sup_{m\geq 1} 
\bignorm{I_{L^p(\tiny\M)}\otimes t_m\colon L^p(\M \otimes M_m)\longrightarrow
L^p(\M\otimes M_m)},
\end{equation}
and
\begin{equation}\label{OpS1Norm}
\bignorm{I^{op}\colon L^p(\M)\longrightarrow 
L^p(\M^{op})}_{S^1} =\,\sup_{m\geq 1} 
\bignorm{I_{L^p(\tiny\M)}\otimes t_m\colon L^p(\M;\;S^1_m)\longrightarrow
L^p(\M;S^1_m)}.
\end{equation}

When $\M=M_n$, the above identities can be more specific. In fact, as we show below, we have that for any $n\geq1$,
\begin{align}\label{special1}
	\bignorm{I^{op}\colon S^p_n\longrightarrow 
\{S^p_n\}^{op}}_{cb}=\bignorm{t_n:S^p_n\to S^p_n}_{cb}\end{align}
and
\begin{align}\label{special2}
\bignorm{I^{op}\colon S^p_n\longrightarrow 
\{S^p_n\}^{op}}_{S^1}=\bignorm{t_n:S^p_n\to S^p_n}_{S^1}.
\end{align}
Using (\ref{OpTrCb}) applied to $\M=M_n$, to prove \eqref{special1}, it is enough to show that 
for any $[x_{ij}]_{1\leq i,j\leq m}$ in $S^p_n\otimes S^p_m$,
\begin{equation}\label{CBTrans}
\bignorm{[t_n(x_{ij})]}_{S^p_m[S^p_n]}
= \bignorm{[x_{ji}]}_{S^p_m[S^{p}_n]}.
\end{equation}
This follows from the fact that $t_m\otimes t_n = t_{nm}$
is an isometry on $S^p_m[S^p_n]\simeq S^p_{nm}$, and 
hence 
$$
\bignorm{(t_m\otimes t_n)[t_n(x_{ij})]}_{S^p_m[S^p_n]} =
\bignorm{[t_n(x_{ij})]}_{S^p_m[S^p_n]}.
$$
Since $(t_m\otimes t_n)[t_n(x_{ij})]= [x_{ji}]$, this yields 
(\ref{CBTrans}). 

Likewise, using (\ref{OpTrS1}) applied to $\M=M_n$, to prove \eqref{special2}, it is enough to show that for any $[x_{ij}]_{1\leq i,j\leq m}$ in $S^p_n\otimes S^1_m$, 
\begin{equation}\label{RegTrans}
\bignorm{[t_n(x_{ij})]}_{S^p_n[S^1_m]}
= \bignorm{[x_{ji}]}_{S^{p}_n[S^1_m]}.
\end{equation}
Assume that $\norm{[t_n(x_{ij})]}_{S^p_n[S^1_m]}<1$.
According to \eqref{defS1Norm}, we can write
$$
t_n(x_{ij})= \sum_k a_{ik}b_{kj}
$$
for some $a_{ik},b_{kj}$ in $S^{2p}_n$ such that 
$\sum_{i,k} a_{ik} a_{ik}^*$ and $\sum_{k,j} b_{kj}^*b_{kj}$
have norm $<1$ in $S^p_n$. Then we have
$$
x_{ij} = \sum_k t_n(a_{ik}b_{kj})\,=\sum_k t_n(b_{kj}) t_n(a_{ik}),
$$
hence
$$
x_{ji} = \sum_k t_n(b_{ki}) t_n(a_{jk}).
$$
Further
$$
\sum_{k,j} t_n(a_{jk})^* t_n(a_{jk})
=t_n\Bigl(\sum_{j,k} a_{jk}a_{jk}^*\Bigr),
$$
and $t_n$ is an isomerty on $S^p_n$. Consequently,
$\sum_{k,j} t_n(a_{kj})^* t_n(a_{jk})$
has norm $<1$ in $S^p_n$. Similarly,
$\sum_{i,k} t_n(b_{ki}) t_n(b_{ki})^*$
has norm $<1$ in $S^p_n$. This shows that
$\norm{[x_{ji}]}_{S^p_n[S^1_m]}<1$. 
We have thus proved the
inequality $\geq$ in (\ref{RegTrans}). Reversing the argument
we find the other inequality.

In the sequel, $E(x)$ denotes the integer part of $x$.
 
\begin{lemma}\label{l32}
Suppose that $\mathcal{M}$ is a semifinite von Neumann algebra.

$(i)$ If $\mathcal{M}$ is subhomogeneous of degree $\leq N$ for some $N\geq1$, then 
for all $[x_{ij}]\in M_m\otimes L^p(\mathcal{M})$, $m\geq1$, we have that
\begin{align*}
\|[x_{ji}]\|_{L^p(\mathcal{M}\otimes M_m)}\leq N^{2\lvert 1/2-1/p\rvert}\|[x_{ij}]\|_{L^p(\mathcal{M}\otimes M_m)},
\end{align*}
and
\begin{align*}
\|[x_{ji}]\|_{L^p(\mathcal{M}; S_m^1)}\leq N\|[x_{ij}]\|_{L^p(\mathcal{M};\; S_m^1)}.
\end{align*}

$(ii)$ Suppose that there exists $K\geq1$ such that for all $[x_{ij}]\in L^p(\mathcal{M}) \otimes S^p_m$, $m\geq1$,
\begin{align}\label{eq5}
\|[x_{ji}]\|_{L^p(\mathcal{M}\otimes M_m)}\leq K\|[x_{ij}]\|_{L^p(\mathcal{M} \otimes M_m)}.
\end{align}
Then if $p\neq2$, $\mathcal{M}$ is subhomogeneous of degree $\leq N$  with $\displaystyle{N=E\left(K^{\frac{1}{2\lvert 1/2-1/p\rvert}}\right)}$.

$(iii)$ Suppose that there exists $K\geq1$ such that for all $[x_{ij}]\in L^p(\mathcal{M}) \otimes S^p_m$, $m\geq1$,
\begin{align}\label{eq5*}
\|[x_{ji}]\|_{L^p(\mathcal{M};S^1_m)}\leq K\|[x_{ij}]\|_{L^p(\mathcal{M};\;S^1_m)}.
\end{align}
Then $\mathcal{M}$ is subhomogeneous of degree $\leq N$  with $N=E(K)$.
\end{lemma}

\begin{proof}
$(i)$ Assume that $\M=L^\infty(\Omega;M_n)$. Let $m\geq1$ be given. We have that
\begin{align*}
 L^p(\M \otimes M_m)\simeq L^p(\Omega; S^p_m[S^p_{n}]).
\end{align*}
By Pisier-Fubini Theorem \cite[(3.6)]{P5},
\begin{align*}
 L^p(\M;S^1_m)\simeq L^p(\Omega; S^p_{n}[S^1_m]).
 \end{align*}
Consequently,
\begin{align}\label{n1}
\bignorm{I_{L^p(\tiny\M)}\otimes t_m\colon L^p(\M\otimes M_m)\longrightarrow
L^p(\M\otimes M_m)} =\bignorm{t_m\otimes I_{S^p_n}\colon S^p_m[S^p_n]
\longrightarrow
S^p_m[S^p_n]}.
\end{align}
and
\begin{align}\label{n2}
\bignorm{I_{L^p(\tiny\M)}\otimes t_m\colon L^p(\M;S^1_m)\longrightarrow
L^p(\M;S^1_m)} = \bignorm{I_{S^p_n}\otimes t_m\colon S^p_n[S^1_m]
\longrightarrow
S^p_n[S^1_m]}.
\end{align}
Applying (\ref{OpCbNorm}) to both sides of (\ref{n1}), we deduce
$$
\bignorm{I^{op}\colon L^p(\M)\longrightarrow 
L^p(\M^{op})}_{cb} =\bignorm{I^{op}\colon S^p_n\longrightarrow 
\{S^p_n\}^{op}}_{cb},
$$
and applying (\ref{OpS1Norm}) to both sides of (\ref{n2}), we deduce that
$$
\bignorm{I^{op}\colon L^p(\M)\longrightarrow 
L^p(\M^{op})}_{S^1} =\bignorm{I^{op}\colon S^p_n\longrightarrow 
\{S^p_n\}^{op}}_{S^1}.
$$

By \cite[Lemma 5.3]{LMZ2},
$$
\bignorm{t_n:S^p_n\to S^p_n}_{cb}=n^{2\lvert 1/p-1/2\rvert}\quad\text{and}\quad \bignorm{t_n:S^p_n\to S^p_n}_{S^1}=n,
$$
hence we obtain by \eqref{special1} and \eqref{special2} that 
$$
\bignorm{I^{op}\colon L^p(\M)\longrightarrow 
L^p(\M^{op})}_{cb} =n^{2\lvert1/p-1/2\rvert}\quad\text{and}\quad \bignorm{I^{op}\colon L^p(\M)\longrightarrow 
L^p(\M^{op})}_{S^1}=n.
$$

When $\M$ is subhomogeneous of degree $\leq N$, there exist $r\geq1$, integers $1\leq n_1\leq n_2\leq \dots \leq n_r\leq N$ and abelian von Neumann algebras $L^\infty(\Omega_1),\dots,L^\infty(\Omega_r)$ such that \eqref{n-sub} holds. Then for any $m\geq1$, we have that
 \begin{align*}
 L^p(\M\otimes M_m)\simeq \mathop\oplus\limits^p_{1\leq j\leq r}L^p(\Omega_j; S^p_m[S^p_{n_j}])\quad\text{and}\quad
 L^p(\M;S^1_m)\simeq \mathop\oplus\limits^p_{1\leq j\leq r}L^p(\Omega_j; S^p_{n_j}[S^1_m]).
 \end{align*}
Using our previous argument and direct sums we deduce that
$$
\|I^{op}:L^p(\M)\to L^p(\M^{op})\|_{cb}\leq N^{2\lvert\frac{1}{p}-\frac{1}{2}\rvert}\quad\text{and}\quad \|I^{op}:L^p(\M)\to L^p(\M^{op})\|_{S^1}\leq N.
$$
The result follows from \eqref{OpTrCb} and \eqref{OpTrS1}.

$(ii)$ Suppose that $\M$ is not subhomogeneous of degree $\leq N=E(K^{\frac{1}{2\lvert1/2-1/p\rvert}})$. By Lemma \ref{l2}, there exists a complete isometry
$$
S^p_{N+1}\,\hookrightarrow\, \M.
$$
This embedding implies that for any $m\geq 1$,
$$
\bignorm{t_m\otimes I_{S^p_{N+1}}\colon
S^p_{m}[S^p_{N+1}] \longrightarrow S^p_{m}[S^p_{N+1}]}
\leq \bignorm{I_{L^p(\tiny\M)}\otimes t_m\colon
L^p(\M\otimes M_m) \longrightarrow L^p(\M\otimes M_m)}.
$$
According to (\ref{OpCbNorm}) and \eqref{special1}, this implies that
$$
\bignorm{t_{N+1}\colon 
S^p_{N+1}\longrightarrow S^p_{N+1}}_{cb}\leq
\bignorm{I^{op}\colon L^p(\M)\longrightarrow 
L^p(\M^{op})}_{cb}.
$$
Hence 
$$
\bignorm{I^{op}\colon L^p(\M)\longrightarrow 
L^p(\M^{op})}_{cb}\geq  (N+1)^{2 \vert \frac{1}{p} -\frac12 \vert}.
$$
Comparing this with inequality (\ref{eq5}) above and applying \eqref{OpTrCb}, we get a contradiction. 

$(iii)$ The proof is similar to the proof of part $(ii)$.
\iffalse
%Likewise, if $\M$ is not subhomogeneous of degree $\leq N$, then 
using the existence of an $S^1$-isometry from 
$S^p_{N+1}$ into $\M$, given by Lemma \ref{l2}, we obtain that
$$
\bignorm{I^{op}\colon L^p(\M)\longrightarrow 
L^p(\M^{op})}_{S^1}\geq N+1.
$$
Comparing this with inequality (\ref{eq5*}) above, we get that  $N+1$ can be maximum $E(K)$ and therefore $\M$ is subhomogeneous. \fi
\end{proof}

\begin{prop}\label{p34}
Let $T:L^p(\mathcal{M})\to L^p(\mathcal{N})$ be separating. If $\M$ is subhomogeneous then $T$ is completely bounded and $S^1$-bounded.
\end{prop}

\begin{proof}
Changing $T$ to $w^*T$, we can assume that $w=J(1)$. By \cite[Remark 4.3]{LMZ2}, we can write $T$ as a sum $T=T_1+T_2$ such that $T_1$ has a direct Yeadon type factorization and $T_2$ has an anti-direct Yeadon type factorization. By Propositions \ref{4.4} and \ref{4.5}, $T_1$ is completely bounded and $S^1$-bounded. Hence it suffices to show that $T_2$ is completely bounded and $S^1$-bounded. Let $I^{op}:L^p(\M)\to L^p(\M^{op})$ be the identity map and set $T_2^{op}=:T_2\circ {I^{op}}^{-1} $. Since $T_2$ has an anti-direct Yeadon type factorization, $T_2^{op}$ has a direct Yeadon type factorization. So, by Propositions \ref{4.4} and \ref{4.5}, $T_2^{op}$ is completely bounded and $S^1$-bounded. Since $\M$ is subhomogeneous, part $(i)$ of Lemma \ref{l32} and its proof show that $I^{op}$ is completely bounded and $S^1$-bounded. By composition, we obtain that $T_2=T_2^{op}\circ I^{op}$ is completely bounded and $S^1$-bounded.
\end{proof}

%%%%%%
\begin{prop}\label{p33}
Suppose that $T:L^p(\mathcal{M})\to L^p(\mathcal{N})$ is a bijective separating map with an anti-direct Yeadon type factorization. 
\begin{itemize}
\item[(i)] If $p\neq2$ and $T$ is completely bounded then $\mathcal{M}$ is subhomogeneous. 
\item[(ii)] If $T$ is $S^1$-bounded then $\mathcal{M}$ is subhomogeneous. 
\end{itemize}
\end{prop}

\begin{proof}
$(i)$ Suppose that $T:L^p(\M)\to L^p(\N)$, $1\leq p\neq2<\infty$, is a bijective separating map with an anti-direct Yeadon type factorization. Assume that $T$ is completely bounded. Let $I^{op}:L^p(\M)\to L^p(\M^{op})$ be the identity map and set $T^{op}:=T\circ {I^{op}}^{-1}$. Since $T$ is bijective with an anti-direct Yeadon type factorization, $T^{op}$ is bijective with a direct Yeadon type factorization. By part $(i)$ of Proposition \ref{lastprop} and Remark \ref{remarkJ}, ${T^{op}}^{-1}$ is also separating with a direct Yeadon type factorization. Therefore, by Proposition \ref{4.4}, ${T^{op}}^{-1}$ is completely bounded. Hence, $I^{op}:= {T^{op}}^{-1}\circ T$ is completely bounded. It now follows from part $(ii)$ of Lemma \ref{l32} and \eqref{OpTrCb} that $\M$ is subhomogeneous.

$(ii)$ The same argument as in part $(i)$ with $S^1$-bounded (norm) replacing completely bounded (norm), Proposition \ref{4.5} replacing Proposition \ref{4.4}, part $(iii)$ of Lemma \ref{l32} replacing its part $(ii)$ and \eqref{OpTrS1} replacing \eqref{OpTrCb} yields the result.
\end{proof}
\begin{remark}
Suppose that $T:L^p(\M)\to L^p(\N)$, $1\leq p<\infty$, is a surjective separating isometry with an anti-direct Yeadon type factorization. The proof of Proposition \ref{p33} shows that when $T$ is completely bounded and $p\neq2$, $\M$ is subhomogeneous of degree $\leq E(\|T\|_{cb}^{\frac{1}{2\lvert 1/2-1/p\rvert}})$. When $T$ is $S^1$-bounded, $\M$ is subhomogeneous of degree $\leq E(\|T\|_{S^1})$. 
\end{remark}
\begin{theorem}\label{4.8}
Let $T : L^p(\M) \to L^p(\N)$, $1\leq p\neq2 <\infty$, be a bounded separating map that is surjective. Then the following are equivalent.
\begin{itemize}
\item[(i)] $T$ is completely bounded.
\item[(ii)] There exists a decomposition $\M=\M_1 \mathop\oplus\limits^{\infty}\M_2$ such that $T|_{L^p(\tiny\M_1)}$ has a direct Yeadon type factorization and $\M_2$ is subhomogeneous.
\end{itemize}
\end{theorem}

\begin{proof}
$(i)\implies (ii)$ Suppose that $T:L^p(\mathcal{M})\to L^p(\mathcal{N})$, $1\leq p \neq2 <\infty$, is a surjective completely bounded separating map. In view of Lemma \ref{4.1}, we may assume $T$ is bijective. 

By Proposition \ref{3.1}, there exist decompositions $\M=\M_1\mathop{\oplus}\limits^\infty\M_2$ and $\N=\N_1\mathop{\oplus}\limits^\infty\N_2$ and surjective separating maps $T_1:L^p(\M_1)\to L^p(\N_1)$ and $T_2:L^p(\M_2)\to L^p(\N_2)$ such that $T_1$ has a direct Yeadon type factorization, $T_2$ has an anti-direct Yeadon type factorization and $T=T_1+T_2$. Since $T$ is completely bounded, $T_2$ is also completely bounded.  By part $(i)$ of Proposition \ref{p33}, $\mathcal{M}_2$ must be subhomogeneous.

$(ii) \implies (i)$ This is a consequence of Propositions \ref{4.4} and \ref{p34}.
\end{proof}
\begin{theorem}\label{4.9}
Let $T : L^p(\M) \to L^p(\N)$, $1\leq p <\infty$, be a separating map that is surjective. Then the following are equivalent.
\begin{itemize}
\item[(i)] $T$ is $S^1$-bounded.
\item[(ii)] There exists a decomposition $\M=\M_1 \mathop\oplus\limits^{\infty}\M_2$ such that $T|_{L^p(\tiny\M_1)}$ has a direct Yeadon type factorization and $\M_2$ is subhomogeneous.
\end{itemize}
\end{theorem}
\begin{proof}
	The proof is similar to Theorem \ref{4.8}, replacing completely bounded with $S^1$-bounded, part $(i)$ of Proposition \ref{p33} by its part $(ii)$ and Proposition \ref{4.4} by Proposition \ref{4.5}.
\end{proof}
The following example shows the surjectivity assumption in Theorems \ref{4.8} and \ref{4.9} is essential. In fact in this example, on a non-subhomogeneous semifinite von Neumann algebra $\M$ and for a given $\epsilon>0$, we construct a separating isometry $T:L^p(\M)\to L^p(\N)$ such that $T$ is not surjective, $\|T\|_{cb}\leq 1+\epsilon$, $\|T\|_{S^1}\leq 1+\epsilon$ and part $(ii)$ of Theorems \ref{4.8} and \ref{4.9} is not satisfied. 

The isometry $T$ in our example is set up between hyperfinite von Neumann algebras and so $\|T\|_{cb}\leq\|T\|_{S^1}$ (see \cite[Proposition 2.2]{P2} and \cite[Proposition 3.11]{LMZ2}). Therefore, we only need to verify that for such $T$ we have that $\|T\|_{S^1}\leq 1+\epsilon$.
\begin{example}\label{lastexample}
Let $1<p<\infty$. Consider the von Neumman algebra
$$
\M=\ell^\infty \{M_n\}= \bigl\{(x_n)_{n\geq 1}\, :\,\forall\, n\geq 1, \ 
x_n\in M_n\,\ \hbox{and}\,\ \sup_{n\geq 1}\norm{x_n}_\infty\,<\infty\bigr\},
$$
the infinite direct sum of all $M_n$, $n\geq 1$.
Let 
$
\N:=\M \mathop\oplus\limits^\infty \M,
$
the direct sum of two copies of $\M$. The noncommutative $L^p$-space associated with $\M$ is 
$$
\ell^p \{S^p_n\}= \bigl\{(x_n)_{n\geq 1}\, :\,\forall\, n\geq 1,\ 
x_n\in S^p_n\,\  \hbox{and}\,\ \sum_{n\geq 1}\norm{x_n}_p^p\,<\infty\bigr\},
$$
equipped with the norm
$$
\bignorm{(x_n)_{n\geq 1}}_p=\Bigl(\sum_{n=1}^{\infty}\norm{x_n}_p^p\Bigr)^{\frac{1}{p}},
$$
and so the noncommutative $L^p$-space associated with $\N$ is 
$
\ell^p\{S^p_n\}\mathop\oplus\limits^p \ell^p\{S^p_n\}.
$
Let $(\beta_n)_{n\geq 1}$ be a sequence in the interval $(0,1)$. 
We may define two operators 
$$
T_1\colon \ell^p\{S^p_n\}\to \ell^p\{S^p_n\}
\qquad \hbox{and}\qquad T_2\colon \ell^p\{S^p_n\}\to \ell^p\{S^p_n\}
$$
by setting
$$
T_1\bigl((x_n)_{n\geq 1}\bigr)\,=\, \bigl((1-\beta_n)^{\frac{1}{p}} x_n\bigr)_{n\geq1}
\qquad \hbox{and}\qquad 
T_2\bigl((x_n)_{n\geq 1}\bigr)\,=\, \bigl(\beta_n^{\frac{1}{p}}\, t_n( x_n)\bigr)_{n\geq1}
$$
for any $x=(x_n)_{n\geq 1}\in\ell^p\{S^p_n\}$. 
Consider
$$
T\colon \ell^p\{S^p_n\}\to \ell^p\{S^p_n\}\mathop{\oplus}\limits^p
\ell^p\{S^p_n\},\qquad T(x)=(T_1(x),T_2(x)).
$$
It is plain that $T$ is an isometry. Indeed for any $x=(x_n)_{n\geq 1}\in\ell^p\{S^p_n\}$, we have
\begin{align*}
\norm{T(x)}^p_p & = \norm{T_1(x)}^p_p +\norm{T_2(x)}^p_p
\\ 
& = \sum_{n=1}^{\infty}(1 -\beta_n)\norm{x_n}^p_p
\, +\, \sum_{n=1}^{\infty} \beta_n \norm{{}^t x_n}^p_p
= \sum_{n=1}^{\infty} \norm{x_n}^p_p \, =\norm{x}^p_p.
\end{align*}

Given $\epsilon>0$,  consider the above construction with 
$$
\beta_n =\frac{(1+\epsilon)^p-1}{n^p-1}.
$$
We show that $T$ is $S^1$-bounded with $\|T\|_{S^1}\leq 1+\epsilon$. Indeed consider an integer $m\geq 1$. We have 
$$
\ell^p \{S^p_n\} \bigl[S^1_m\bigr]\,=\, \ell^p \{S^p_n[S^1_m]\},
$$
and therefore, we also have that
$$
\left(\ell^p\{S^p_n\}\mathop{\oplus}\limits^p \ell^p\{S^p_n\}\right) \bigl[S^1_m\bigr]\,=\,
\ell^p\{S^p_n[S^1_m]\}\mathop{\oplus}\limits^p \ell^p\{S^p_n[S^1_m]\}.
$$
Now let $x=(x_n)_{n\geq 1}\in \ell^p \{S^p_n[S^1_m]\}$ (here each 
$x_n$ is an element of $S^p_n[S^1_m])$. Then 
$$
(I_{S^1_m}\otimes T)(x) = \Bigl(
\bigl((1-\beta_n)^{\frac{1}{p}} x_n\bigr)_{n\geq 1},
\bigl(\beta_n^{\frac{1}{p}}\, (t_n\otimes I_{S^1_m})(x_n)\bigr)_{n\geq 1}\Bigr).
$$
Consequently,
\begin{align*}
\bignorm{(I_{S^1_m}\otimes T)(x) }^p_p & = 
\sum_{n=1}^{\infty}(1 -\beta_n)\norm{x_n}^p_{S^p_{n}[S^1_m]}\, +\, \sum_{n=1}^{\infty} \beta_n \norm{(t_n\otimes I_{S^1_m})(x_n)}^p_{S^p_{n}[S^1_m]}
\\ 
& \leq \sum_{n=1}^{\infty}(1 -\beta_n)\norm{x_n}^p_{S^p_n[S^1_m]} \, +\, n^p\beta_n\norm{x_n}_{S^p_n[S^1_m]}^p\qquad\text{by \cite[Lemma 5.3 (ii)]{LMZ2}}\\
& \leq(1+\epsilon)^p\sum_{n=1}^{\infty} \norm{x_n}^p_{S^p_n[S^1_m]} =(1+\epsilon)^p\norm{x}_p^p.
\end{align*}

It is clear that $T$ is separating and that the 
Jordan homomorphism $J\colon\M\to\N$
in its Yeadon triple is given by
$$
J\bigl((x_n)_{n\geq 1}\bigr) = \bigl((x_n)_{n\geq 1},
(t_n(x_n))_{n\geq 1}\bigr).
$$
It follows that whenever $\M_1$ is a non zero summand of $\M$, the Yeadon factorization of the restriction of $T$ to $L^p(\M_1)$ is neither direct nor indirect. A fortiori, $T$ does not satisfy the assertion $(ii)$ of Theorem \ref{4.9}.
\end{example}

%%%%%%%%%%%%%%%%%%%%%%%%%%%%%%%%%%%%
\section*{acknowledgment}
The second author would like to gratefully thank ``Laboratoire de Math\'ematiques de Besan\c con" for hospitality and excellent working condition she received during her visit.

%%%%%%%%%%%%%%%%%%%%%%%%%%%%%%%%%%%%%%%%%%%%%%%%%%%%%%%%
\nocite{*}

\bibliographystyle{alpha}

\end{document}